\documentclass[11pt]{amsart}
\usepackage[top=1.2in, bottom=1.5in, left=1.5in, right=1.5in]{geometry}

\usepackage{amsfonts, amssymb, amsmath, enumerate}

\numberwithin{equation}{section}

\DeclareMathOperator{\E}{\mathbb{E}}

\DeclareMathOperator*{\diag}{diag}
\DeclareMathOperator*{\Span}{span}

\DeclareMathOperator*{\tr}{tr}

\DeclareMathOperator*{\Det}{det}
\DeclareMathOperator*{\conv}{conv}

\def \C {\mathbb{C}}
\def \N {\mathbb{N}}
\def \P {\mathbb{P}}
\def \R {\mathbb{R}}

\def \K {\mathbb{K}}

\def \e {\varepsilon}
\def \d {\delta}
\def \l {\lambda}

\def \HS {\mathrm{HS}}
\def \tr {\mathrm{tr}}
\def \vol {{\rm vol}}
\def \Id {{\rm Id}}

\def \etc {,\ldots,}

\def \ii {{\mathbf i}}

\newcommand{\norm}[1]{\left \| #1 \right \|}
\newcommand{\pr}[2]{\left \langle {#1} , {#2} \right \rangle}

\newtheorem{theorem}{Theorem}[section]
\newtheorem{proposition}[theorem]{Proposition}
\newtheorem{corollary}[theorem]{Corollary}
\newtheorem{lemma}[theorem]{Lemma}

\theoremstyle{remark}
\newtheorem{remark}[theorem]{Remark}

\begin{document}

\title[On the volume of non-central sections of a cube]{On the volume of non-central sections \\
of a cube}

\author{Hermann K\"onig}
\address[Hermann K\"onig]{Mathematisches Seminar\\
Universit\"at Kiel\\
24098 Kiel, Germany
}
\email{hkoenig@math.uni-kiel.de}

\author[Mark Rudelson]{ Mark Rudelson$^*$}
\address[Mark Rudelson]{
Department of Mathematics\\
University of Michigan\\
Ann Arbor, MI 48109-1043
}
\email{rudelson@umich.edu}

\thanks{
  $^*$Research was supported in part by NSF DMS grant 1807316.}

\keywords{Volume, non-central section, hypercube}
\subjclass[2000]{Primary: 52A38, 52A40. Secondary: 52A20}

\begin{abstract}
Let $Q_n$ be the cube of side length one centered at the origin in $\R^n$, and let $F$  be an affine
$(n-d)$-dimensional subspace of $\R^n$ having distance to the origin less than or equal to $\frac 1 2$, where $0<d<n$.
We show that the $(n-d)$-dimensional volume of the section $Q_n \cap F$ is bounded below by a value
$c(d)$ depending only on the codimension $d$ but not on the ambient dimension $n$ or a particular subspace $F$.
In the case of hyperplanes, $d=1$, we show that $c(1) = \frac{1}{17}$ is a possible choice. We also consider a complex
analogue of this problem for a hyperplane section of the polydisc.
\end{abstract}

\maketitle

\section{Introduction and main results}

The hyperplane conjecture, a.k.a. slicing problem and the Busemann-Petty problem gave a new impetus to study the volume of sections of convex bodies by linear subspaces and to estimate these quantities. The Busemann-Petty problem has been solved, and the reader may consult Koldobsky's book \cite{K}. In the case of the slicing problem only partial results are known, cf. Bourgain \cite{Bo}, Milman and Pajor \cite{MiP}, Klartag \cite{Kl}. For a comprehensive description of results related to the slicing problem see the book \cite{BGVV}. The study of sections of convex bodies is a very active area with applications in functional analysis, probability and computer science. \\

It is difficult to find the volume of maximal sections or minimal central sections of specific convex bodies by hyperplanes or, more generally, by $k$-codimensional subspaces. This is true even for classical bodies like the $l_p^n$-balls, $0 < p < \infty$, or the regular $n$-simplex. For the $n$-cube, the maximal hyperplane section was found by Ball \cite{B} in a celebrated paper in 1986. He extended his result to $k$-codimensional sections of the $n$-cube in Ball \cite{B1}. The best lower estimate for central cubic sections was known before, cf. Hadwiger \cite{Ha}, Hensley \cite{He} and Vaaler \cite{Va}. For central slabs of the cube of small width, optimal lower estimates were given by Barthe and Koldobsky \cite{BK}. The case of the cube is important for many problems in the area where it is a conjectured extremal case or provides a counterexample. Ball's result immediately provided a counterexample to the Busemann-Petty problem in high dimensions. Further, the cube has maximal volume ratio among all symmetric convex bodies, as shown by Ball \cite{B2}. \\

Concerning the $n$-simplex, Webb \cite{We} determined the maximal section through the centroid, using techniques of Ball. 
In the case of the $l_p^n$-balls, $0<p<\infty$, the minimal $k$-codimensional central sections for $p \ge 2$ and the maximal $k$-codimensional sections for $p \le 2$ were established by Meyer and Pajor \cite{MeP}. The minimal central hyperplane section for $0<p \le2$ is the one perpendicular to the main diagonal, as shown by Koldobsky \cite{K1}. The case of the maximal hyperplane section for $p \ge 2$ is open. In this case, the perpendicular direction of the maximal hyperplane section of $l_p^n$ has to depend both on $p$ and $n$, as shown by Oleszkiewicz \cite{O1}. \\

The results mentioned so far concern central sections through the origin in the case of symmetric convex bodies or through the centroid for general convex bodies. For non-central sections, not too many results are known. Moody, Stone, Zach and Zvavitch \cite{MSZZ} showed that the maximal hyperplane section of the $n$-cube of unit volume at almost maximal possible distance from the origin, namely between $\sqrt{n-1}/2$ and $\sqrt n /2$, is the one perpendicular to the main diagonal. They also solved the case of non-central line sections of the cube. For small dimensions, $n=2, 3$, the extremal hyperplane sections for all possible distances from the origin were calculated in K\"onig and Koldobsky \cite{KK1}.\\

In this paper, we establish non-trivial lower estimates for non-central sections of the $n$-cube by $k$-codimensional sections in the situation where the distance from the origin guarantees a non-void intersection. We also consider the complex case of the polydisc in $\C^n$. To formulate our results precisely, we start with a few definitions. \\

Consider a cube  of a {\it unit volume} in the space $\K^n$, where $\K \in \{\R,\C \}$.
As mentioned, the  sections of the cube by linear subspaces are classical objects of study in convex geometry, and precise estimates of their maximal and minimal volume are known.
Namely, let $\norm{\cdot}_\infty$ and $|\cdot |$ denote the supremum and the euclidean norm on $\K^n$, respectively, where $\K \in \{\R,\C \}$. For volume calculations, we identify $\C^n$ with $\R^{2n}$ and use the volume there. Let
\[
 Q_n :=\{x\in \K^n\mid \norm{x}_\infty \le \alpha\}
\]
be the $n$-dimensional cube (polydisc) of volume 1, i.e. $\alpha=1/2$ if $\K=\R$ and $\alpha=1/\sqrt{\pi}$ if $\K=\C$.
In the real case,  for any linear subspace of $E \subset \R^n$ of dimension $n-d$,
\[
 1 \le \vol_{n-d}(Q_n \cap E) \le 2^{d/2}.
\]
The lower estimate is due to Vaaler \cite{Va}, and the upper one to Ball \cite{B1}.
In the complex case, Oleszkiewicz and Pelczy\'nski \cite{OP} proved that for codimension $1$, $1 \le \vol_{2n-2}(Q_n \cap E) \le 2$.
Less is known about the non-central sections which are the subject of the current paper.

Let us discuss the real case first.
Fix a subspace $E \subset \R^n$ and consider sections of the cube by subspaces parallel to $E$. More precisely, for a vector $v \in E^\perp$, consider a function
\[
 \Phi(E, v):  = \vol_{n-d} \big( Q_n \cap (E+v) \big).
\]
  Brunn's theorem asserts that  $\Phi$ is an even function achieving the maximal value at the origin. This, in combination with Ball's theorem, provides an upper bound for the function $\Phi$ for all $E$ and $v$.
  If $|v| > \frac{1}{2}$, then a non-trivial lower bound for this function is impossible to achieve.
  Indeed, if $E$ is orthogonal to one of the basic vectors $e_j$, and $v=t e_j$ with $t> \frac{1}{2}$, then $ Q_n \cap (E+v) = \varnothing$.
  Note that $\Phi(E, t e_j)$ is discontinuous at $t=1$. A discontinuity of this type does not occur in the corresponding case of the $l_p^n$-balls for $0<p<\infty$ since these convex bodies are strictly convex.
  Our first main result provides a non-trivial lower estimate for the volume $\Phi(E, v)$ of the cubic section for all $E$ and $v$ as long as  $|v| \le \frac{1}{2}$.
  Moreover, this estimate is independent of the ambient dimension $n$ and the space $E$.

\begin{theorem} \label{thm: section}
 For any $d \in \N$, there is $\e(d)>0$ such that for any $n>d$ and any $(n-d)$-dimensional affine subspace $F \subset \R^n$ whose distance to the origin is smaller than or equal to $1/2$,
 \[
    \vol_{n-d}(Q_n \cap F) \ge \e(d).
 \]
\end{theorem}
As the discussion above shows, the distance $1/2$ is the maximal possible one.

The  value of the bound $\e(d)$ can be traced from the proof of Theorem \ref{thm: section}.
We believe, however, that this value is quite far from the best possible. A better bound can be obtained for the sections of codimension $1$, i.e., whenever $d=n-1$.
We will present this bound in the unified way for both real and complex scalars.

To this end, let us introduce some notation.
 Given a vector $a\in \K^n$ of length $|a|=1$ and $t\in \K$, we introduce the {\em hyperplane section} of the cube
$$S(a,t):=\{x\in \K^n\mid \norm{x}_\infty \le \alpha,\ \langle x,a\rangle=\alpha t\} = Q_n \cap H$$
where $H = \{\alpha t \cdot a\} + a^\perp$, and its volume
$$A(a,t):= A_\K(a,t) := \left\{\begin{array}{c@{\quad}l}
\vol_{n-1}(S(a,t)) \; \ , & \; \K=\R \\
\vol_{2n-2}(S(a,t)) \; , & \; \K=\C
\end{array}\right\}.$$
For $a = (a_j)_{j=1}^n \in \K^n$, let $a^*$ denote the decreasing rearrangement of the sequence $(|a_j|)_{j=1}^n$. Since the volume is invariant under coordinate permutations and sign changes (rotation of coordinate discs in the complex case), we have $A(a,t)=A(a^*,|t|)$. Therefore we will assume in the following that $a=(a_j)_{j=1}^n$, $a_j \ge 0$ and $t\ge 0$. \\

By Corollary 5 of K\"onig and Koldobsky \cite{KK3} we have that
$$ A(a,t) \le \sqrt{\frac 2 {1+t^2}} \ , \ \K=\R  \text{ \hspace{0,4cm} and \hspace{0,4cm} } A(a,t) \le \frac 2 {1+t^2} \ , \ \K=\C , $$
so that $A(a,1) \le 1$ always holds.\\

As in the general case, if the distance parameter $t$ is strictly bigger than 1, the non-central hyperplane  $H = \{\alpha t \cdot a\} + a^\perp$ might not intersect $Q_n$ and $A(a,t)$ might be 0.
Assume that $t \in [0,1]$.
Our second main result gives  explicit bounds for $A(a,t)$ which are independent of the dimension $n$ of the cube and of the direction $a$.

\begin{theorem}\label{th1}
Let $a \in \K^n$ with $|a|=1$. Then \\
$$ \frac 1 {17} < 0.06011 < A_{\R}(a,1) \le 1 \ , $$
$$ \frac 1 {27} < 0.03789 < A_{\C}(a,1) \le 1 \ . $$
\end{theorem}

\vspace{0,5cm}
Clearly, the lower bounds are not optimal.
However, they cannot be improved by more than a factor of $\simeq 5.2$ in the real case and by a factor of $\simeq 7.3$ in the complex case, see Remark \ref{rem: diagonal}.

In the rest of the paper, we prove Theorems \ref{thm: section} and \ref{th1}.
The proof of Theorem \ref{thm: section} is contained in Section \ref{sec: multidim}.
In the course of it, we represent the function $\Phi(E,v)$ as the density $f_X$ of the projection of a random vector  uniformly distributed in $Q_n$ onto the space $E^\perp$. We use both the geometric and the probabilistic definition of this function passing several times from one to another throughout the proof. If the space $E$ is almost orthogonal to a coordinate vector and $v$ is almost parallel to it, we derive the desired estimate by analyzing the characteristic function of $f_X$ and using the log-concavity of this density. The analysis of the characteristic function relies in turn on its representation as the difference of characteristic functions of some other sections of the cube. The opposite case splits into two separate subcases. If the vector $v$ is incompressible, i.e., far from any low-dimensional coordinate subspace, we prove the required bound probabilistically. If this vector is compressible, we rely on the previous analysis to reduce the bound to a similar geometric problem but in dimension depending only on $d$. The estimate in this case can be obtained directly. \\

We start preparing the ground for proving Theorem \ref{th1} in Section \ref{sec: formulas}. In this section, we use the Fourier transform to represent the volume of a hyperplane section as a certain integral over the product of $n$ euclidean spheres $S^{k-1}$ with respect to the Haar measure. Here, $k=3$ in the real case, and $k=4$ in the complex case. The estimate of these integrals requires a lower bound for the probability that $|\sum_{j=1}^n a_j U_j| \ge 1$ where $a=(a_1, \ldots, a_n) \in S^{n-1}$ and  $U_1, \ldots, U_n$ are independent random vectors uniformly distributed in $S^{k-1}$. A similar problem with $U_j$ being scalar random variables has been extensively studied because of its importance in computer science, see e.g., \cite{HK, BTNR, O, BH} and the references therein. However, the methods used there do not seem to be suitable to the vector-valued random variables. In Section \ref{sec: Orlicz}, we develop a new method based on estimates of the Laplace transform and duality of Orlicz spaces. This method may be of independent interest as it is applicable to a broader class of random vectors. The probability itself is estimated in Section \ref{sec: tails}. Finally, in Section \ref{sec: hyperplane}, we apply the toolkit created in three previous sections to complete the proof of Theorem \ref{th1}.

\vspace{0,5cm}
{\bf Acknowledgment.} The first author is grateful to S. Kwapie\'n \cite{Kw} for indicating the basic idea of the first proof of Proposition \ref{prop5}, citing ideas which go back to Burkholder \cite{Bu}, and for providing the reference to Veraar's paper \cite{V}. Proposition \ref{prop5} is an important step in the proof of Theorem \ref{th1}.

Part of this work was done when the second author visited Weizmann Institute of Science. He is grateful to the Institute for its hospitality and for the excellent working conditions. He is also thankful to Ofer Zeitouni for helpful discussions.

We also thank A. Koldobsky for discussions on the contents of this paper.

\section{A lower bound for all codimensions} \label{sec: multidim}

In this section, we prove Theorem \ref{thm: section}.
Let $F \subset \R^n$ be an affine subspace whose distance to the origin is $1/2$.
 We will represent $F$ as $F= \frac{1}{2} v+E$, where $E \subset \R^n$ is an $(n-d)$-dimensional linear subspace, and $v \in E^{\perp}, \ |v| = 1$.
 Denote by $P: \R^n \to \R^n$  the orthogonal projection onto $E^{\perp}$.

 The strategy of the proof will depend on the position of the space $E$ and the magnitude of the largest coordinate of $v$.
 We start from the case when $E$ is almost orthogonal to a coordinate vector and $v$ is almost parallel to this vector.

 \begin{lemma} \label{lem: parallel}
   For any $d<n$, there exists $\e_1(d), \delta_1(d)$ such that if $|P e_1| \ge 1-\delta_1(d)$ and $v=\frac{P e_1}{|P e_1|}$, then
   \[
    \vol_{n-d} \left(Q_n \cap \left(\frac{1}{2} v+E \right) \right) \ge \e_1(d).
   \]
 \end{lemma}

 \begin{proof}
 Assume for a moment that $e_1 \perp E$, and thus $v=e_1$. Then $Q_n \cap \left(\frac{1}{2} v+E \right)$ is a central section of the $(n-1)$-dimensional face of $Q_n$ containing $\frac{1}{2}e_1$. In this case,
    \[
    \vol_{n-d} \left(Q_n \cap \left(\frac{1}{2} v+E \right) \right) \ge 1
   \]
 by Vaaler's theorem \cite{Va}. This means that we can assume that $|Pe_1|<1$ for the rest of the proof.

    A random point $\xi \in Q_n$ can be considered as a random vector of density $1$ in the cube.
In this probabilistic interpretation, the volume of the section $ \vol_{n-d}(Q_n \cap (E+u))$ is the density of the random vector $P \xi$ distributed in $E^\perp$ at the point $u \in E^\perp$. It would be more convenient to consider this random vector distributed  in $\R^d$ instead.
To this end, notice that  the singular value decomposition of $P$ yields the existence of a
$d \times n$ matrix $R$ satisfying
\[
P = R^\top R, \quad R R^\top = I_d.
\]
Therefore, $ f_X(u)=\vol_{n-d}(Q_n \cap (E+u))$ can be viewed as the density of the vector $R \xi$ in $\R^d$. We will use the geometric and the probabilistic  interpretation interchangeably throughout the proof.

  The Fourier transform of the random variable $X=R \xi$ can be written as
 \[
   \phi_X(t)= \int_{\R^d} f_X(x) \exp(- \ii \cdot 2 \pi \pr{x}{t}) \, dx
   =\prod_{j=1}^{n} \frac{\sin(\pi \pr{Re_j}{t})}{\pi \pr{Re_j}{t}}
 \]
 for $t \in \R^d$, where we used the normalization by $2 \pi$ for convenience.
 By the Fourier inversion formula,
 \[
  f_X \left(\frac{1}{2}v \right)= \int_{\R^d} \exp(\ii \pi \pr{v}{t}) \phi_X(t) \, dt
  = \frac{1}{\pi^d}  \int_{\R^d} \cos(\pr{v}{t}) \phi_X(t/\pi) \, dt.
 \]
  Hence,
 \begin{align*}
   2f_X \left(\frac{1}{2}v \right)
   & = \frac{2}{\pi^d} \int_{\R^d} \cos (\pr{v}{t}) \prod_{j=1}^{n} \frac{\sin( \pr{Re_j}{t})}{ \pr{Re_j}{t}} \, dt \\
    & =\frac{1}{\pi^d}  \int_{\R^d} \frac{\sin \left( \left(\frac{1}{|Re_1|}+ 1 \right) \pr{Re_1}{t} \right)}{\pr{Re_1}{t}} \prod_{j=2}^{n} \frac{\sin( \pr{Re_j}{t})}{\pr{Re_j}{t}} \, dt \\
    & \qquad - \frac{1}{\pi^d}  \int_{\R^d} \frac{\sin \left( \left(\frac{1}{|Re_1|}- 1 \right) \pr{Re_1}{t} \right)}{ \pr{Re_1}{t}} \prod_{j=2}^{n} \frac{\sin(\pr{Re_j}{t})}{\pr{Re_j}{t}} \, dt.
 \end{align*}
 Define $d \times d$ matrices $\Lambda_+, \Lambda_-: \R^d \to \R^d$ by
 \[
 \Lambda_\pm
 =\left( \left(\frac{1}{|Re_1|}\pm 1 \right)^2 (Re_1)(Re_1)^\top
   + \sum_{j=2}^n (Re_j)(Re_j)^\top \right)^{-1/2}.
 \]
 As $|Re_1|=|Pe_1|<1$, the matrix $\Lambda_-$ is well-defined.
 Then
 \[
  \Lambda_+^{-1}|_{(Re_1)^{\perp}}=\Lambda_-^{-1}|_{(Re_1)^{\perp}}
  =\text{id}|_{(Re_1)^{\perp}},
 \]
 and so
 \begin{align*}
  \Det(\Lambda_+)
  &=\frac{1}{\Det(\Lambda_+^{-1})}
  =\frac{1}{|\Lambda_+^{-1} Re_1|}
  =\left( \left(1+|Re_1| \right)^2 + \sum_{j=2}^{n} \pr{Re_1}{Re_j}^2  \right)^{-1/2} \\
  &= (2 + 2 |Re_1|)^{-1/2}.
  \intertext{Similarly,}
  \Det(\Lambda_-)
  &=\frac{1}{\Det(\Lambda_-^{-1})}
  =\frac{1}{|\Lambda_-^{-1} Re_1|}
  =\left( (1-|Re_1|)^2 + \sum_{j=2}^{n} \pr{Re_1}{Re_j}^2 \right)^{-1/2} \\
  &= (2 - 2 |Re_1|)^{-1/2}.
 \end{align*}
 Using the change of variables in the integrals above, we can write
 \begin{align*}
 2 f_X \left(\frac{1}{2}v \right)
     = & \frac{1}{\pi^d} \left(\frac{1}{|Re_1|}+ 1 \right)  \Det(\Lambda_+) \int_{\R^d} \prod_{j=1}^{n} \frac{\sin(\pr{\theta_j}{t})}{\pr{\theta_j}{t}} \, dt \\
     - &  \frac{1}{\pi^d} \left(\frac{1}{|Re_1|}- 1 \right)  \Det(\Lambda_-) \int_{\R^d} \prod_{j=1}^{n} \frac{\sin(\pr{\eta_j}{t})}{\pr{\eta_j}{t}} \, dt ,
 \end{align*}
 where
 \[
   \begin{cases}
     \theta_1=  (\frac{1}{|Re_1|}+ 1) \Lambda_+ Re_1, \\
     \theta_j=\Lambda_+ Re_j,  \mbox{ for } j>1,
   \end{cases}
   \qquad
   \begin{cases}
     \eta_1= (\frac{1}{|Re_1|}- 1) \Lambda_- Re_1, \\
     \eta_j=\Lambda_- Re_j,  \mbox{ for } j>1,
   \end{cases}
 \]
 Note that
 \[
  \sum_{j=1}^{n} \theta_j \theta_j^\top = \sum_{j=1}^{n} \eta_j \eta_j^\top = I_d.
 \]
 This allows to view
  both integrals above as the volumes of certain sections of $Q_n$ by $(n-d)$-dimensional linear subspaces.
  More precisely,
  \begin{align*}
    \frac{1}{\pi^d}  \int_{\R^d} \prod_{j=1}^{n} \frac{\sin(\pr{\theta_j}{t})}{\pr{\theta_j}{t}} \, dt
    &= \vol_{n-d}(Q_n \cap E_1) \\
    \intertext{and}
    \frac{1}{\pi^d}  \int_{\R^d} \prod_{j=1}^{n} \frac{\sin(\pr{\eta_j}{t})}{\pr{\eta_j}{t}} \, dt
    &= \vol_{n-d}(Q_n \cap E_2)
  \end{align*}
  for some linear subspaces $E_1,E_2 \subset \R^n$.
  This can be easily checked using the Fourier inversion formula as above.
  A theorem of Vaaler \cite{Va} asserts that the volume of any central section of the unit cube is at least $1$, and a theorem of Ball \cite{B1}  states that it does not exceed $(\sqrt{2})^d$.
  Therefore,
 \begin{align*}
   2f_X \left(\frac{1}{2}v \right)
   & \ge \left(\frac{1}{|Re_1|}+ 1 \right)  \Det(\Lambda_+)
  - (\sqrt{2})^d \left(\frac{1}{|Re_1|}- 1 \right)  \Det(\Lambda_-) \\
   & = \frac{1}{\sqrt{2} |Re_1|} \left(1+|Re_1| \right)^{1/2} - \frac{(\sqrt{2})^{d-1}}{|R e_1|} \left(1-|Re_1| \right)^{1/2}
   \ge \e_1(d),
 \end{align*}
 for some $\e_1(d)>0$ whenever $|Pe_1| =|Re_1| \ge 1-\delta_1(d)$ for an appropriately small $\delta_1(d)>0$.
 \end{proof}

 The previous lemma provided a lower bound for the volume of the section if the vector $v$ has the form $\frac{Pe_1}{|Pe_1|}$.
 We will now extend this bound to the vectors which are close to this one.

\begin{lemma}  \label{lem: almost parallel}
  For any $d \in \N$, there exist $\delta_2(d),\e_2(d)$ such that if $|Pe_1| \ge 1-\delta_1(d)$ and $v=\frac{Pe_1}{|Pe_1|}$, then for any $w \in E^\perp$ with $w \perp v, \ |w|<\delta_2(d)$,
  \[
   \vol_{n-d} \left( Q_n \cap \left(\frac{1}{2} v +w+E \right) \right) \ge \e_2(d).
  \]
\end{lemma}

\begin{proof}
  By Lemma \ref{lem: parallel},
  \[
   \vol_{n-d} \left( Q_n \cap \left(\frac{1}{2} v +E \right) \right) \ge \e_1(d).
  \]
   Also, applying the same lemma to the linear subspace $\tilde{E}:=\Span(w,E)$, we get
  \[
   vol_{n-d} \left( Q_n \cap \left(\frac{1}{2} v + \tilde{E} \right) \right) \ge \e_1(d-1).
  \]
  Define the function $h: \R \to \R$ by
  \[
   h(x)=\vol_{n-d} \left( Q_n \cap \left(\frac{1}{2} v +x \frac{w}{|w|}+E \right) \right).
  \]
  Then the previous inequalities read $h(0) \ge \e_1(d), \   \int_{\R} h(x) \, dx \ge \e_1(d-1)$.
  Assume that $h(|w|) \le h(0)/2$. Since the function $h$ is even and log-concave, this implies $h(k|w|) \le 2^{-|k|} h(0)$ for all $k \in \mathbb{Z}$, and hence
  \[
   \e_1(d-1) \le \int_{\R} h(x) \, dx \le 4  |w| \cdot h(0)
   \le  4  |w| \cdot  \vol_{n-d} \left( Q_n \cap E \right) \le 4 |w| \cdot  (\sqrt{2})^d,
  \]
  where we used Ball's theorem \cite{B1}  in the last inequality.
  This means that the statement of the lemma holds with
  \[
   \delta_2(d)= \frac{\e_1(d-1)}{4 (\sqrt{2})^d} \quad \text{and} \quad \e_2(d)=\frac{\e_1(d)}{2}
  \]
  since for $|w| < \delta_2(d)$ we would get a contradiction to our assumption. Thus $h(|w|) > h(0)/2 \ge \e_2(d)$, so the proof is complete.
\end{proof}

We summarize Lemmas \ref{lem: parallel} and \ref{lem: almost parallel} in the following corollary.
\begin{corollary}  \label{cor: almost parallel}
  For any $d \in \N$, there exist $\delta_3(d),\e_3(d)$ such that if $v \in E^\perp, \ |v|=1$ and $\norm{v}_\infty \ge 1-\delta_3(d)$ then
  \[
   \vol_{n-d} \left( Q_n \cap \left(\frac{1}{2} v +E \right) \right) \ge \e_3(d).
  \]
\end{corollary}

\begin{proof}
 Without loss of generality, assume that $v_1=\pr{v}{e_1} \ge 1-\d$, where $\d=\d_3(d)$ will be chosen later.
 Then
 \[
  |Pe_1| \ge \pr{Pe_1}{v}=\pr{e_1}{v} \ge 1-\d,
 \]
 and
 \begin{align*}
  \left| v- \frac{Pe_1}{|Pe_1|} \right|
  &\le |v-Pe_1|+ \left| Pe_1-\frac{Pe_1}{|Pe_1|} \right| \\
  &\le \left( |v|^2-2 \pr{v}{Pe_1} + |Pe_1|^2 \right)^{1/2} + \big( 1- |Pe_1| \big) \\
  &\le (2-2(1-\d))^{1/2} + \delta.
 \end{align*}
 This means that choosing $\d$ small enough, we can ensure that the conditions of Lemma \ref{lem: almost parallel} are satisfied.
\end{proof}

 Let $X=P \xi$, where $\xi$ is a random vector uniformly distributed in $Q_n$. The density $f_X$  of the vector $X$ is even and log-concave, so the set $D:=\{y \in E^\perp: \ f_X(y) \ge f_X(\frac{1}{2} v) \}$ is convex and symmetric.
 We need the following simple lemma which would allow us to reduce the estimate of the density of a  multi-dimensional projection to a bound on a probability of a half-space.

\begin{lemma} \label{lem: to one dim}
  Let
  \[
   D:= \left \{y \in E^\perp: \ f_X(y) \ge f_X \left(\frac{1}{2}v \right) \right \}.
  \]
  Let $S \subset E^\perp$ be a supporting hyperplane to $D$ at $v$ in $E^\perp$, and write $S= \tau u+L$, where $L$ is a linear subspace of $E^\perp$,  $u \in E^\perp \cap S^{n-1}$ satisfies $u \perp L$, and $\tau \ge 0$.
  Then $\tau \le \frac{1}{2}$ and
  \[
    f_X \left(\frac{1}{2}v \right)  \ge \max \left(f_X(\tau u), c(d) (\P(\pr{\xi}{u} \ge \tau))^{1+d/2} \right)
  \]
  for some $c(d)>0$.
\end{lemma}

 \begin{proof}
 The inequalities $\tau \le \frac{1}{2}$ and  $f_X(\frac{1}{2}v) \ge f_X(\tau u)$ follows immediately from $\tau u \in S$ and the convexity of $D$.

   To prove the other inequality, denote $\nu=\P(\pr{\xi}{u} \ge \tau)$ and
    set
    \[
     K:= \left \{y \in E^\perp: \ \pr{y}{u} \ge \tau \text{ and } |y| \le  \sqrt{\frac{d}{\nu}} \right \}.
    \]
    Note that $\E |X|^2 = \E |P \xi|^2= \sum_{j=1}^n |Pe_j|^2 \E \xi_j^2= \frac{d}{12}$.
  Using Markov's inequality, we get
  \[
   \P(X \in K)
   \ge \\P(\pr{X}{u} \ge \tau) - \P \left(|X| \le  \sqrt{\frac{d}{\nu}} \right)
   \ge \nu - \frac{\E |X|^2}{d/\nu}
   \ge \frac{\nu}{2}.
  \]
  For any $y \in K, \ f_X(y) \le f_X(\frac{1}{2}v)$ since $K \subset E^\perp \setminus D$. Therefore
  \[
   f_X \left(\frac{1}{2}v \right) \ge \frac{\P(X \in K)}{\vol_{d}(K)}.
  \]
  As $\vol_d(K) \le \big( \sqrt{d/\nu} \big)^d \vol_d(B_2^d) \le C(d) \nu^{-d/2}$, the lemma follows.
 \end{proof}

To use Lemma \ref{lem: to one dim}, we have to bound $\P(\pr{\xi}{u} \ge \tau)$ for a unit vector $u \in S^{n-1}$.
This bound is obtained differently depending on whether the vector $u$ is close to a low-dimensional space. We consider the case when it is far from such spaces, i.e., it has enough mass supported on small coordinates.
The opposite case will be considered in the proof of Theorem \ref{thm: section}.
\begin{lemma} \label{lem: B-E}
 Let $u \in S^{n-1}$, and let $\xi$ be a random vector uniformly distributed in $Q_n$.
 For any $\e>0$, there exist $\delta, \eta>0$ such that if $J_\delta= \{j: \ |u_j|<\delta \}$ and $\sum_{j \in J_\delta} u_j^2>\e^2$  then
 \[
  \P(\pr{\xi}{u} \ge 1) \ge \eta.
 \]
\end{lemma}
\begin{proof}
  Denote $Y=\sum_{j \in J_\delta}  \xi_j u_j, \ Z=\sum_{j \notin J_\delta}  \xi_j u_j$, where $\xi_1 \etc \xi_n$ are i.i.d. random variables uniformly distributed in $[-\frac{1}{2}, \frac{1}{2}]$.

  Let $g$ be the standard normal random variable.
  By the Berry-Esseen theorem,
  \begin{align*}
    \P( \sum_{j \in J_\delta} \xi_j u_j \ge 1)
    & \ge \P(\sqrt{\sum_{j \in J_\delta}  u_j^2} \cdot g  \ge 1) - c \max_{j \in J_\delta} \frac{|u_j|}{\sqrt{\sum_{j \in J_\delta}  u_j^2}} \\
     & \ge \P \left(g \ge \frac{1}{\e} \right)-c \frac{\delta}{\e}
     \ge \tilde{\eta} >0
  \end{align*}
  if $\delta=\delta(\e)$ is chosen sufficiently small.
  Hence,
  \[
    \P(\pr{\xi}{u} \ge 1)
     \ge \P \left( Y \ge 1 \text{ and }Z \ge 0 \right)
      \ge  \tilde{\eta} \cdot \frac{1}{2}=:\eta,
  \]
  since $Y$ and $Z$ are independent.
  The lemma is proved.
\end{proof}

Having proved these lemmas, we can derive  Theorem \ref{thm: section}.
\begin{proof}[Proof of Theorem \ref{thm: section}]
Recall that
\[
 \vol_{n-d} \left( Q_n \cap \left( \frac{1}{2} v +E \right) \right)
  =  f_X \left(\frac{1}{2}v \right),
 \]
 where $X=P \xi$ and $\xi$ is a random vector uniformly distributed in $Q_n$.

 Let $u$ and $\tau$ be as in Lemma \ref{lem: to one dim}.
   Take
   \[
    \e= \sqrt{\frac{\e_3(d)}{2} }
    \]
    and choose the corresponding $\delta$ from Lemma \ref{lem: B-E}.
Define  $J_\delta$ as in this lemma.
If $\sum_{j \in J_\delta} u_j^2 \ge \e^2$, then by Lemmas \ref{lem: to one dim} and \ref{lem: B-E},
\[
 f_X \left(\frac{1}{2}v \right) \ge c(d) \big( \P(\pr{\xi}{u} \ge \tau) \big)^{1+d/2}
 \ge c(d) \eta^{1+d/2}
\]
as $\tau \in [0,\frac{1}{2}]$.

 Assume now that $\sum_{j \in J_\delta} u_j^2 \le \e^2$.
 If $\norm{  u}_\infty \ge 1-  \e_3(d)$, then the statement of the theorem follows from Corollary \ref{cor: almost parallel} since $f_X(\frac{1}{2} v) \ge f_X(\tau u) \ge f_X (\frac{1}{2} u)$.
 Thus, we can assume that
 \begin{equation}\label{eq: in Q}
     \norm{  u}_\infty \le 1 - \e_3(d).
 \end{equation}
 We will use the inequality
 \[
 f_X \left(\frac{1}{2}v \right)
   \ge c(d) \big( \P(\pr{\xi}{u} \ge \tau) \big)^{1+d/2}
  \ge c(d) \left( \P \left(\pr{\xi}{u} \ge \frac{1}{2} \right) \right)^{1+d/2}
 \]
  again.
 This shows that to prove the theorem, it is enough to bound $\P \left(\pr{\xi}{u} \ge \frac{1}{2} \right)$ from below by a quantity depending only on $d$.

 Decompose $\pr{\xi}{u}=Y+Z$ where $Y=\sum_{j \in J_\delta}  \xi_j u_j, \ Z=\sum_{j \notin J_\delta}  \xi_j u_j$ as above. Then
\begin{align*}
 \P \left(\pr{\xi}{u} \ge \frac{1}{2} \right)
  & \ge \P \left(Z \ge \frac{1}{2} \text{ and } Y \ge 0 \right)
  = \frac{1}{2} \P \left(Z \ge \frac{1}{2} \right) \\
   & = \frac{1}{2} \P \left(\sum_{j \notin J_\delta} \xi_j w_j \ge \theta \right)
\end{align*}
 where
 \[
  w_j = \frac{u_j}{\sqrt{\sum_{j \notin J_\delta} u_j^2}}
  \quad \text{and} \quad \theta = \frac{1}{2 \sqrt{\sum_{j \notin J_\delta} u_j^2}}.
 \]
 Note that
 \[
  k:=|[n] \setminus J_\delta|= |\{j \in [n]: \ |u_j| \ge \delta \}|
  \le \delta^{-2},
 \]
 where $\delta$ depends only on $d$.
 To simplify the notation, assume that $[n] \setminus J_\d=[k]$.
 We can recast $\P (\sum_{j \notin J_\delta} \xi_j w_j \ge \theta)$ as
 \[
  \P (\sum_{j \notin J_\delta} \xi_j w_j \ge \theta)
  = \vol_k(Q_k \cap (w^{\perp}_+ + \theta w)),
 \]
 where $Q_k=[-\frac{1}{2}, \frac{1}{2} ]^k$,  $w \in S^{k-1}$ is the vector with coordinates $w_j, \ j \in [k]$, and $w^{\perp}_+ = \{y \in  \R^k: \ \pr{y}{w} \ge 0 \}$ is a half-space orthogonal to $w$.
 Previously, we reformulated a geometric problem of bounding the volumes of non-central sections of the cube in a probabilistic language. Here, we reduce it back to a similar geometric problem but in dimension $k$ which depends only on $d$ and codimension $1$.

 By our assumption,
 \[
  \theta
  \le \frac{1}{2 \sqrt{1- \e^2}},
 \]
 and, in view of  \eqref{eq: in Q}, we have
 \[
  \frac{\sqrt{1-\e^2}}{2(1-\e_3(d) ) } w  \in Q_k.
 \]
  Therefore,
\begin{align*}
  &\vol_{k} (Q_k \cap (w^{\perp}_+ + \theta w)) \\
  &\ge \vol_{k} \left( \conv \left(Q_k \cap w^\perp,  \frac{\sqrt{1-\e^2}}{2(1-\e_3(d) )} w \right) \cap (w^{\perp}_+ +  \theta w) \right) \\
  &= \left(1- \theta \left(  \frac{\sqrt{1-\e^2}}{2(1-\e_3(d) )} \right)^{-1} \right)^{k} \cdot
  \vol_{k} \left( \conv \left(Q_k \cap w^\perp,  \frac{\sqrt{1-\e^2}}{2(1-\e_3(d) )} w \right)  \right)   \\
  &\ge  \left(1-  \frac{1- \e_3(d)}{1-  \e^2} \right)^{k} \cdot
  \frac{1}{k!}  \left(  \frac{\sqrt{1-\e^2}}{2(1-\e_3(d) )}  \right)^k \vol(Q _k \cap w^\perp)  \\
  &\ge \left(1- \frac{1- \e_3(d)}{1-  \e^2}  \right)^{k} \cdot
  \frac{1}{k!}  \left(  \frac{\sqrt{1-\e^2}}{2(1-\e_3(d) )}  \right)^k,
\end{align*}
where the last inequality follows from Vaaler's theorem \cite{Va}.
Recalling that $\e=\sqrt{\e_3(d)/2}$ and $k$ depends only on $d$, we see that the quantity above is positive and depends only on $d$ as well.
This completes the proof of the theorem.
\end{proof}

\section{Volume formulas} \label{sec: formulas}

To prove Theorem \ref{th1}, we start with the following known volume formulas.

\begin{proposition}\label{prop2}
Let $a\in \R_+^n, \ |a|=1, \ t \ge 0$. Then
\begin{equation}\label{eq1}
A_\R(a,t) = \frac 2\pi \ \int_0^\infty \ \prod_{j=1}^n \ \frac{\sin(a_js)}{a_js} \ \cos (ts) \ ds \ ,
\end{equation}
\begin{equation} \label{eq2}
A_\C(a,t) = \frac 1 2 \ \int_0^\infty \ \prod_{j=1}^n \ j_1 (a_js) \; J_0(ts) \  s \ ds \ , \ j_1(x):=2 \frac{J_1(x)} x \ .
\end{equation}
Here $J_0$ and $J_1$ denote the standard Bessel functions.
\end{proposition}

Formula \eqref{eq1} whose multidimensional version was used in the previous section can be found in Ball's paper \cite{B} on cubic sections, equation \eqref{eq2} in Oleszkiewicz and Pelczy\'nski \cite{OP}. The case $t=0$ of \eqref{eq1} goes back to P\'olya \cite{P}. A Fourier analytic proof of Proposition \ref{prop2} is outlined in K\"onig and Koldobsky \cite{KK1}, \cite{KK2}. \\

Due to the oscillating character of the integrands in \eqref{eq1} and \eqref{eq2}, it is difficult to find non-trivial lower bounds for $A(a,t)$ using these equations. Therefore we first prove different formulas for $A(a,t)$.

\begin{proposition}\label{prop3}
Let $(\Omega,\P)$ be a probability space and $U_j : \Omega \to S^{k-1} \subset \R^k$, $j=1,\cdots,n$ be a sequence of independent, random vectors uniformly distributed on the sphere $S^{k-1}$, where $k=3$ if $\K=\R$ and $k=4$ if $\K=\C$. Then for any $a\in \R_+^n, \ |a|=1$ and  $t \ge 0$
$$(a) \quad A_\R(a,t) = \int \limits_{|\sum_{j=1}^n a_j U_j| \ge t} \; \frac {d  \P}{|\sum_{j=1}^n a_j U_j|} \; , $$
$$(b) \quad A_\C(a,t) = \int \limits_{|\sum_{j=1}^n a_j U_j| \ge t} \; \frac {d  \P}{|\sum_{j=1}^n a_j U_j|^2} \; . $$
\end{proposition}

\begin{proof}
 (a) Let $m$ denote the normalized Lebesgue surface measure on $S^{k-1} \subset \R^k$ for $k \in \N$, $k \ge 2$.
Then for any fixed vector $e \in S^{k-1}$
$$\int_{S^{k-1}} \exp(it<e,u>) \ dm(u) = \frac {\int_0^\pi \cos(t \cos(\phi)) \ \sin(\phi)^{k-2} \ d \phi}{\int_0^\pi \sin(\phi)^{k-2} \ d \phi} = j_{\frac k 2-1}(t) \ , $$
$j_{\frac k 2 -1}(t) = 2^{\frac k 2 -1} \Gamma(\frac k 2) \frac{J_{\frac k 2 -1}(t)}{t^{\frac k 2 -1}}$, $t>0$. Again, $J_{\frac k 2 -1}$ denote the standard Bessel functions of index $\frac k 2 -1$. In particular, for $k=3$ and $k=4$
\begin{equation}\label{eq3}
\int_{S^2} \exp(it<e,u>) \ dm(u) = \frac {\sin(t)} t \; , \; \int_{S^3} \exp(it<e,u>) \ dm(u) = j_1(t) \ .
\end{equation}
We may assume that $a$ has at least two non-zero coordinates $a_j$ since otherwise the formulas in (a) and (b) just state $1=1$ if $t \le 1$ and $0=0$ if $t>1$. By \eqref{eq3}
\begin{equation}\label{eq4}
\prod_{j=1}^n \frac{\sin(a_j s)}{a_j s} = \int_{(S^2)^n} \exp(is<e,\sum_{j=1}^n a_j u_j>) \ dm(u_1) \cdots dm(u_n) \ .
\end{equation}
This is $O(\frac 1 {s^2})$ as $s \to \infty$, therefore Lebesgue-integrable on $(0,\infty)$. Since \eqref{eq4} holds for all $e \in S^2$, we may integrate over $e$. Using \eqref{eq3} again, we find
$$\prod_{j=1}^n \frac{\sin(a_j s)}{a_j s} = \int_{(S^2)^n} \frac {\sin(|\sum_{j=1}^n a_j u_j| s)}{|\sum_{j=1}^n a_j u_j| s} \ dm(u) \ , \
dm(u) := \prod_{j=1}^n dm(u_j) \ . $$
The factor $|\sum_{j=1}^n a_j u_j|$ results from the necessary normalization $\frac {\sum_{j=1}^n a_j u_j}{|\sum_{j=1}^n a_j u_j|} \in S^2$. Hence, using Proposition \ref{prop2},
\begin{align}\label{eq5}
A_{\R}(a,t) &= \frac 2 \pi \int_0^\infty \left( \int_{(S^2)^n} \ \frac {\sin(|\sum_{j=1}^n a_j u_j| s)}{|\sum_{j=1}^n a_j u_j| s} \ \cos(ts) \ dm(u) \right) \ ds \notag \\
&= \int_{(S^2)^n} \left( \frac 2 \pi \int_0^\infty \ \frac {\sin(|\sum_{j=1}^n a_j u_j| s)}{|\sum_{j=1}^n a_j u_j| s} \ \cos(ts) \ ds \right) \ dm(u) \notag \\
&= \int \limits_{(S^2)^n , \ |\sum_{j=1}^n a_j u_j| \ge t} \; \frac {dm(u)}{|\sum_{j=1}^n a_j u_j|} \ ,
\end{align}
using that
\begin{equation}\label{eq6}
\frac 2 \pi \int_0^\infty \frac{\sin(As)}{As} \ \cos(ts) \ ds =  \left\{\begin{array}{c@{\quad}l}
0 \quad , & \; 0 < A < t \\
\frac 1 A \quad , & \; A > t > 0
\end{array}\right\} .
\end{equation}
Note that $m(|\sum_{j=1}^n a_j u_j| = t) = 0$ since $a$ has at least two non-zero coordinates. The integral in \eqref{eq6} is only conditionally convergent, which requires justification of interchanging the order of integration in \eqref{eq5}. This is allowed if
\begin{equation}\label{eq6a}
\lim_{N \to \infty} \int_{(S^2)^n} \left( \frac 2 \pi \int_N^\infty \ \frac {\sin(|\sum_{j=1}^n a_j u_j| s)}{|\sum_{j=1}^n a_j u_j| s} \ \cos(ts) \ ds \right) \ dm(u) = 0
\end{equation}
is shown. We have in terms of the Sine integral $Si$, $Si(x):=\frac 2 \pi \int_0^x \frac {\sin(t)} t \ dt, \ x \in \R$ that
$$\frac 2 \pi \int_N^\infty \frac{\sin(As)}{As} \ \cos(ts) \ ds =  \left\{\begin{array}{c@{\quad}l}
\frac{Si((t-A)N)-Si((t+A)N)}{\pi A} \quad \; , & \; 0 < A < t \\
\frac{\pi -Si((A-t)N)-Si((A+t)N)}{\pi A} \; , & \; A > t > 0
\end{array}\right\} $$
and hence
\begin{align*}
& \int_{(S^2)^n} \left( \frac 2 \pi \int_N^\infty \ \frac {\sin(|\sum_{j=1}^n a_j u_j| s)}{|\sum_{j=1}^n a_j u_j| s} \ \cos(ts) \ ds \right) \ dm(u)  \\
& =
  \int \limits_{(S^2)^n, |\sum_{j=1}^n a_j u_j| < t} \frac 1 {\pi |\sum_{j=1}^n a_j u_j|} \ \left[ Si\left((t-|\sum_{j=1}^n a_j u_j|)N \right) \right.
 \\
 & \hskip 2.5 in - \left. Si\left((t+|\sum_{j=1}^n a_j u_j|)N \right) \right] \ dm(u)
\\
& + \int \limits_{(S^2)^n, |\sum_{j=1}^n a_j u_j| > t} \frac 1 {\pi |\sum_{j=1}^n a_j u_j|} \ \left[ \pi - Si\left((|\sum_{j=1}^n a_j u_j|-t)N \right) \right. \\
& \hskip 2.5in  -  \left. Si\left((t+|\sum_{j=1}^n a_j u_j|)N \right) \right] \ dm(u)
\end{align*}
Since for all $b \in S^2$ and $\beta > 0$
$$\int_{S^2} \frac {dm(u_1)}{|b+\beta u_1|} = \frac 1 2 \int_{-1}^1(|b|^2+\beta^2+2\beta |b| v)^{-\frac 1 2} \ dv =
\left\{\begin{array}{c@{\quad}l}
\frac 1 {|b|} \quad , & \; 0 < \beta < |b| \\
\frac 1 \beta \quad , & \; \beta > |b| > 0
\end{array}\right\}  < \infty \ $$
and since the $Si$-function is bounded in modulus by 2, the two integrands involving the $Si$-function are bounded in modulus by an integrable function independent of $N$. Since for any $c>0$ we have that $\lim_{N \to \infty} Si(cN) = \frac \pi 2$, the integrands converge to $0$ pointwise. By the Lebesgue theorem, \eqref{eq6a} follows and \eqref{eq5} is proven. \\

(b) Using \eqref{eq3}, we find similarly as in (a)
$$\prod_{j=1}^n \ j_1(a_j s) = \int_{(S^3)^n} \ j_1(|\sum_{j=1}^n a_j u_j| s) \ dm(u) \; , \; dm(u) := \prod_{j=1}^n dm(u_j) , $$
and by Proposition \ref{prop2}
\begin{align}\label{eq7}
A_{\C}(a,t) &= \frac 1 2  \int_0^\infty \left( \int_{(S^3)^n} \ j_1(|\sum_{j=1}^n a_j u_j| s) \ J_0(ts) \ dm(u) \right) \ s \ ds \notag \\
&= \int_{(S^3)^n} \left( \frac 1 2 \int_0^\infty \ j_1(|\sum_{j=1}^n a_j u_j| s) \ J_0(ts) \ s \ ds \right) \ dm(u) \notag \\
&= \int_{(S^3)^n} \left( \int_0^\infty \ J_1(|\sum_{j=1}^n a_j u_j| s) \ J_0(ts) \ ds \right) \ \frac {dm(u)}{|\sum_{j=1}^n a_j u_j|} \notag \\
&= \int \limits_{(S^3)^n , \ |\sum_{j=1}^n a_j u_j| \ge t} \; \frac {dm(u)}{|\sum_{j=1}^n a_j u_j|^2} \ ,
\end{align}
since by  Gradstein, Ryshik \cite{GR}, 6.51.
\begin{equation}\label{eq8a}
\int_0^\infty \ J_1(As) \ J_0(ts) \ ds = \left\{\begin{array}{c@{\quad}l}
0 \quad , & \; 0 < A < t \\
\frac 1 A \quad , & \; A > t > 0
\end{array}\right\} ,
\end{equation}
which is a conditionally convergent integral. To justify exchanging the order of integration in \eqref{eq7},
we employ the product formula for Bessel functions
$$J_0(u) \ J_0(v) =  \frac 1 \pi \int_0^\pi J_0(\sqrt{u^2+v^2+2uv\cos(\phi)} \ ) \ d \phi \; , \; u, v \in \R \ ,$$
cf. Watson \cite{W}, 11.1. Since $J_0' = -J_1$, differentiating this with respect to $u$, inserting $u=As$, $v=ts$ and integrating with respect to $s$ yields that for all $N, A, t >0$
\begin{align*}
&\left| \int_0^N J_1(As)  J_0(ts) \ ds \right| \\
& = \left|\frac 1 \pi \int_0^\pi \left(\int_0^N J_1(\sqrt{A^2+t^2+2At \cos(\phi)} s)  ds \right)
\frac{A+t \cos(\phi)}{\sqrt{A^2+t^2+2At \cos(\phi)}}  d \phi \right| \\
& = \left|\frac 1 \pi \int_0^\pi \left(1-J_0(\sqrt{A^2+t^2+2At \cos(\phi)}N)\right)  \frac{A+t \cos(\phi)}{A^2+t^2+2At \cos(\phi)}  d \phi   \right| \\
& \le \frac 2 \pi \int_0^\pi \frac{|A+t \cos(\phi)|}{A^2+t^2+2At \cos(\phi)} \ d \phi =: I(A,t) \ ,
\end{align*}
where we also used that $|J_0| \le 1$ holds. Since
$$\int \frac{A+t \cos(\phi)}{A^2+t^2+2At \cos(\phi)} \ d \phi = \frac 1 A \left( \frac \phi 2 + \arctan ( \frac{A-t}{A+t} \tan(\frac \phi 2)) \right) =:\Psi(\phi) \ , $$
we find for $A>t$ that $I(A,t)= \frac 2 A$, $I(A,A)= \frac 1 A$ and for $A<t$ that $I(A,t) = \frac 2 \pi 2 \Psi(\phi_0) \le \frac 2 A$ where $\cos(\phi_0) = - \frac A t$. Thus $I(A,t) \le \frac 2 A$ which implies using \eqref{eq8a}
$$\left| \int_N^\infty J_1(As) \ J_0(ts) \ ds \right| \le \frac 3 A \ . $$
Moreover $\lim_{N \to \infty} \int_N^\infty J_1(As) \ J_0(ts) \ ds = 0$ pointwise and $\int_{(S^3)^n} \frac{dm(u)}{|\sum_{j=1}^n a_j u_j|^2} < \infty$, so that we find similarly as in part (a)
$$\lim_{N \to \infty} \int_{(S^3)^n} \left( \int_N^\infty J_1(|\sum_{j=1}^n a_j u_j| s) \ J_0(ts) \ ds \right) dm(u) = 0 \ , $$
and \eqref{eq7} follows. We basically replaced the $Si$-function in part (a) by $\int_0^x J_1(t) \ dt = 1 - J_0(x)$. \\
Formulas \eqref{eq5} and \eqref{eq7} yield a concrete realization of the formulas in Proposition \ref{prop3} involving independent, uniformly distributed random vectors on $S^{k-1}$ for $k=3,4$.
\end{proof}

\section{Exponential estimates and Orlicz spaces duality} \label{sec: Orlicz}

To prove Theorem \ref{th1}, we use lower estimates for the probability that certain quadratic forms of random variables on spheres $S^{k-1}$ are non-negative.
In this section, we develop a new method of estimating such probabilities. The estimate itself will be obtained in the next section.
Our bound relies on the estimate of  the norm of the quadratic form in the Orlicz space whose Orlicz function is of an exponential type. The Orlicz function we use is close to the $\psi_1$ function used in the large deviations theory. The lower bound on probability is obtained in terms of the norm of the indicator function in the dual of this Orlicz space.

We start with a simple lemma showing that a random vector uniformly distributed over the sphere is subgaussian.
\begin{lemma} \label{lem: subgauss}
  Let $U$ be a random vector uniformly distributed in $S^{k-1}$.
  Then the vector $U$ is $(1/\sqrt{k})$-subgaussian, i.e.,
  \[
   \forall y \in \R^{k} \quad
   \E \exp(\pr{U}{y}) \le \E \exp \left( \frac{|y|^2}{\sqrt{k}}g \right),
  \]
  where $g \in \R$ denotes the standard normal random variable.
\end{lemma}

\begin{proof}
  Due to the rotational invariance, we can assume that $y= \l e_1$ for some $\l>0$.
  Notice that for any $p \in \N$,
  \begin{equation} \label{eq: cos}
   \E \pr{U}{e_1}^{2p} \le k^{-p} \E g^{2p}.
  \end{equation}
  Indeed, denoting by $g^{(k)}$ the standard Gaussian vector in $\R^k$, we can write
  \[
   \E g^{2p} = \E \pr{g^{(k)}}{e_1}^{2p}= \E |g^{(k)}|^{2p} \cdot \E \pr{U}{e_1}^{2p}
  \]
  and $\E |g^{(k)}|^{2p}\ge (\E |g^{(k)}|^{2})^p=k^p$ by Jensen's inequality.
  Decomposing $e^{\l x}$ into Taylor series and using \eqref{eq: cos}, we derive that
  \[
  \E \exp(\l \pr{U}{e_1})
  \le \E \exp \left( \frac{\l}{\sqrt{k}} g \right).
  \]
  The result follows.
 \end{proof}

The next lemma provides an estimate of the Laplace transform of the relevant quadratic form.
\begin{lemma} \label{lem: Laplace}
  Let $U_1, \ldots, U_n$ be i.i.d. random vectors uniformly distributed in $S^{k-1}$.
  Let $a=(a_1, \ldots, a_n) \in S^{n-1}$, and define
  \[
   S:=\sum_{1 \le i < j \le n} a_i a_j \pr{U_i}{U_j}.
  \]
   Then for any $\l \in (-\sqrt{k/2},\sqrt{k/2})$,
  \[
   \E \exp \left(  \l \frac{S}{(\E S^2)^{1/2} } \right)
   \le \left( 1- \frac{2 \l^2 }{k} \right)^{-k/2}.
  \]\end{lemma}

\begin{proof}

  To simplify the notation, let us estimate $\E \exp (\l S)$.
  Since $S$ is a quadratic form of subgaussian vectors $U_1, \ldots, U_n$, such estimate can be derived from the Hanson-Wright inequality, see Hanson and Wright \cite{HW} and Rudelson and Vershynin \cite{RV}. However, the bound obtained in this way would be too loose for our purposes. Instead, we will use the specific information about this quadratic form to obtain a tighter bound.

Our argument is based on a Laplace transform estimate as in \cite{RV}.
  Let $g_1^{(k)}, \ldots, g_n^{(k)}$ be independent standard Gaussian vectors in $\R^k$.
  By Lemma \ref{lem: subgauss}
    \[
  \E \exp( \pr{U_n}{y})
  \le \E \exp \left( \frac{1}{\sqrt{k}} \pr{g_n^{(k)}}{y} \right).
  \]
  Using this inequality  with fixed $U_1, \ldots, U_{n-1}$, we get
  \begin{align*}
    \E \exp(\l S)
     & = \E\exp \left( \l \sum_{1 \le i<j \le n-1} a_i a_j \pr{U_i}{U_j} + \pr{\l \sum_{1 \le i<j \le n-1} a_i a_j U_i}{U_n} \right) \\
     & \le \E\exp \left( \l \sum_{1 \le i<j \le n-1} a_i a_j \pr{U_i}{U_j} + \pr{\l \sum_{1 \le i<j \le n-1} a_i a_j U_i}{\frac{g_n^{(k)}}{\sqrt{k}}} \right),
  \end{align*}
  Repeating the same argument for other $U_j$, we obtain
  \begin{align*}
   \E \exp(\l S)
   &\le \E \exp \left( \frac{\l}{k}\sum_{1 \le i<j \le n} a_i a_j \pr{g_i^{(k)}}{g_j^{(k)}} \right) \\
   &= \left[ \E \exp \left( \frac{\l}{k}\sum_{1 \le i<j \le n} a_i a_j g_i g_j \right) \right]^k,
  \end{align*}
    where $g_1, \ldots, g_n$ are i.i.d. $N(0,1)$ random variables.
    To derive the last equality, we notice that $ \pr{g_i^{(k)}}{g_j^{(k)}}$ is the sum of $k$ i.i.d. random variables distributed like $g_i g_j$.
    The previous inequality can be rewritten as
  \begin{equation}  \label{eq: pow k}
   \E \exp(\l S)
   \le  \left[ \E \exp \left( \frac{\l}{2k} (g^{(n)})^\top B g^{(n)} \right) \right]^k,
  \end{equation}
  where $g^{(n)}=(g_1, \ldots, g_n) \in \R^n$ is the standard Gaussian vector, and
  $B$ is a symmetric $n \times n$ matrix with the entries $b_{i,j}=a_i a_j$ when $i \neq j$ and $0$ otherwise, i.e.,
  \[
   B=  a a^\top - \diag (a_1^2, \ldots, a_n^2).
  \]
    Denote the eigenvalues of $B$ by $\mu_1 \ge \cdots \ge \mu_n$.
  Then by interlacing $\mu_1> 0 \ge \mu_2 \ge \cdots \ge \mu_n$. Also,
    \[
  \mu_1 \le \norm{B}_{\HS} = \left(  \sum_{i \neq j} a_i^2 a_j^2 \right)^{1/2}, \quad \text{and} \quad
   \sum_{j=1}^n \mu_j  =\tr (B)=0.
  \]
  By the rotational invariance, we have
  \[
   \E \exp \left( \frac{\l}{2k} (g^{(n)})^\top B g^{(n)} \right)
   = \E \exp \left( \frac{\l}{2k} \sum_{j=1}^n \mu_j g_j^2 \right)
   =\prod_{j=1}^n \left( 1- \frac{\l \mu_j}{k} \right)^{-1/2}
  \]
  provided that $ \frac{\l \mu_j}{k}<1$ for all $j \in [n]$.
  Since
  \begin{equation}  \label{eq: mu_1}
   |\mu_j| \le \norm{B}_{\HS} = \sqrt{2k} \left( \frac{1}{k} \sum_{1 \le i <j \le n} a_i^2 a_j^2 \right)^{1/2}
   = \sqrt{2k} \left( \E S^2 \right)^{1/2},
  \end{equation}
  this restriction is satisfied if we assume that
  \begin{equation} \label{eq: rst}
   \frac{\sqrt{2}|\l|}{\sqrt{k}}  \left( \E S^2 \right)^{1/2}
  <1.
  \end{equation}

  Assume that this restriction holds.
  Recall  that $0 \ge \mu_2 \ge \cdots  \ge \mu_n$,  and $ \sum_{j=2}^n \mu_j  =-\mu_1$. Applying the inequality
  \[
   \prod_{j=2}^n (1+y_j) \ge 1+ \sum_{j=2}^n y_j
  \]
   valid for all $y_2, \ldots y_n \in (-1,1)$ having the same sign, we derive that
  \begin{align*}
  \prod_{j=1}^n \left( 1- \frac{\l \mu_j}{k} \right)
  &\ge  \left( 1- \frac{\l \mu_1}{k} \right) \cdot \left( 1- \sum_{j=2}^n \frac{\l \mu_j}{k}  \right)
  = \left( 1- \frac{\l \mu_1}{k} \right) \cdot \left( 1+ \frac{\l \mu_1}{k} \right) \\
  &= 1- \left( \frac{\l \mu_1}{k}  \right)^2.
  \end{align*}
  In combination with \eqref{eq: mu_1}, this yields
  \[
   \E \exp \left( \frac{\l}{2k} (g^{(n)})^\top B g^{(n)} \right)
   \le \left( 1- \left( \frac{\l \mu_1}{k}  \right)^2 \right)^{-1/2}
   \le  \left( 1- \frac{\l^2 \cdot 2 \E S^2}{k}   \right)^{-1/2}.
  \]
  Taking into account \eqref{eq: pow k}, it shows that if \eqref{eq: rst} holds, then
  \[
  \E \exp(\l S) \le  \left( 1- \frac{\l^2 \cdot 2 \E S^2}{k}   \right)^{-k/2}.
  \]
  The result follows if we replace $\l$ by $\l/(\E S^2)^{1/2}$ in the inequality above.
\end{proof}

We now use the duality of Orlicz norms to estimate the probability that $S>0$.

\begin{lemma} \label{lem: Orlicz}
Let $Y$ be a real-valued random variable.
  Let $\l \in (0,1)$, and let $0<q \le \E Y_+$. Assume that $\E \exp(\l Y_+) < \infty$.Then
  \[
   \P(Y>0)
   \ge \left[ \left(t+\frac{1}{\l q} \right) \log \left( 1+ \frac{1}{\l q t} \right) - \frac{1}{\l q} \right]^{-1}
  \]
  for any
  \[
   0<t \le \frac{1}{\E \exp(\l Y_+)-\l q -1}.
  \]
\end{lemma}

\begin{proof}
Let $t>0$.
  Define the functions $L,M: (0,\infty) \to (0,\infty)$ by
  \[
   L(x)=t(e^x-x-1), \quad M(x)=(t+x) \log \left( 1+ \frac{x}{t} \right)-x, \quad x \in (0,\infty).
  \]
  Then $L$ and $M$ are Orlicz functions. Denote by $\norm{\cdot}_L$ the norm in the Orlicz space $X_L$. Then the dual norm is $\norm{\cdot}_M$.

  If $t$ satisfies the assumption of the lemma, then $\E L(\l Y_+) \le 1$, and so
  \[
  \norm{Y_+}_L \le \frac{1}{\l}.
  \]
  Hence, by duality of Orlicz norms,
  \[
   q
   \le \E ( Y_+ \cdot \mathbf{1}_{(0,\infty)} )
   \le \norm{Y_+}_L \cdot \norm{\mathbf{1}_{(0,\infty)}(Y)}_M
   \le \frac{1}{\l} \cdot \norm{\mathbf{1}_{(0,\infty)}(Y)}_M,
  \]
  or $\norm{\frac{1}{\l q} \mathbf{1}_{(0,\infty)}(Y)}_M \ge 1$.
  This inequality reads
  \[
   1 \le \norm{\frac{1}{\l q} \mathbf{1}_{(0,\infty)}(Y)}_M
   = \left[ \left(t+\frac{1}{\l q} \right) \log \left( 1+ \frac{1}{\l q t} \right) - \frac{1}{\l q} \right] \cdot \P(Y>0),
  \]
  which proves the lemma.
\end{proof}

In order to apply Lemma  \ref{lem: Orlicz}, we need an upper bound for $\E \exp(\l Y_+)$. This is our next task.

\begin{lemma} \label{lem: plus bound}
 Let $Y$ be a real-valued random variable such that $\E Y=0$.
 Then for any $\l>0$,
 \[
  \E \exp(\l Y_+) \le \E \exp(\l Y) + \frac{1}{4}\E \exp(-\l Y).
 \]
\end{lemma}

\begin{proof}
 Denote $p=\P(Y \ge 0)$. Then
  \[
  \E \exp(\l Y_+) = \E \exp(\l Y)  +(1-p) - \E [\exp(\l Y) \mathbf{1}_{(-\infty,0)}(Y)].
 \]
 We can estimate the last term from below by Cauchy-Schwarz inequality:
 \begin{align*}
  (1-p)^2
   &= \left( \E \mathbf{1}_{(-\infty,0)}(Y) \right)^2 \\
  &\le \E [ \exp(\l Y) \mathbf{1}_{(-\infty,0)}(Y) ] \cdot \E [ \exp(-\l Y) \mathbf{1}_{(-\infty,0)}(Y) ] \\
   &\le \E [ \exp(\l Y) \mathbf{1}_{(-\infty,0)}(Y) ] \cdot \E \exp(-\l Y).
 \end{align*}
 This implies that
 \[
  \E \exp(\l Y_+) \le \E \exp(\l Y)+1-p - (1-p)^2 \big( \E \exp(-\l Y) \big)^{-1}.
 \]
 The proof finishes by maximizing this expression over $p \in \R$.
\end{proof}

\begin{remark}
If $\E \exp(-\l Y) >2$, the maximum of the function above is attained outside of the interval $[0,1]$. In this case one can obtain a better bound
 \[
  \E \exp(\l Y_+) \le \E \exp(\l Y) - \big( \E \exp(-\l Y) \big)^{-1} +1
 \]
 by taking $p=0$. However, we are not going to use this improvement.
\end{remark}

\section{Tail estimates} \label{sec: tails}

To estimate $A(a,1)$ from below, we need a lower estimate of $\P(|\sum_{j=1}^n a_j U_j| \ge 1)$ and tail estimates for the random vectors $\sum_{j=1}^n a_j U_j$ in Proposition \ref{prop3}. \\

\begin{proposition}\label{prop5}
Let $(U_j)_{j=1}^n$ be a sequence of independent random vectors uniformly distributed on the sphere
$S^{k-1} \subset \R^k$ for $k \ge 2$. Let $a \in \R_+^n$, $|a|=1$. Then
$$\P(|\sum_{j=1}^n a_j U_j|\ge 1) \ge \frac{2 \sqrt 3 - 3}{3+\frac 4 k} =: \gamma_k \ . $$
For $k=3, 4$ we have the better numerical estimates
$$\P(|\sum_{j=1}^n a_j U_j|\ge 1) \ge 0.1268 \quad ,  \quad k=3 \ , $$
$$\P(|\sum_{j=1}^n a_j U_j|\ge 1) \ge 0.1407 \quad ,  \quad k=4 \ . $$
\end{proposition}

\begin{remark} (a) In the case of the Rademacher variables $(r_j)_{j=1}^n$, Oleszkiewicz \cite{O} showed that
$$\P(|\sum_{j=1}^n a_j r_j| \ge 1) \ge \frac 1 {10} $$
holds. His beautiful scalar proof does not seem to generalize to our case of spherical variables. \\

(b) The estimate of Proposition \ref{prop5} is not optimal. It is unclear whether the minimum occurs for
$a^n = \frac 1 {\sqrt n} (1, \cdots , 1)$. In the Rademacher case, $k=1$, this is not true, as Zhubr showed around 1995 for $n=9$ (unpublished).
For $k \ge2$ and $n \to \infty$, $a^n$ yields that no better lower bound than the following is possible:
By the central limit theorem
$$\frac 1 {\sqrt n} \sum_{j=1}^n U_j \to \mathcal{N}(0,\Sigma) \; , \; \Sigma = \frac 1 k \Id_k $$
with density function $f(x) = \left( \frac k {2 \pi} \right)^{\frac k 2} \exp(- \frac k 2 |x|^2 )$. Therefore
\begin{align*}
\lim_{n \to \infty} \P(\frac 1 {\sqrt n} |\sum_{j=1}^n U_j| \ge 1) &= |S^{k-1}| \left(\frac k {2\pi}\right)^{\frac k 2} \int_1^\infty r^{k-1} \exp(-\frac k 2 r^2) \ dr \\
& = \frac 1 {\Gamma(\frac k 2)} \int_{\frac k 2}^\infty s^{\frac k 2 -1} \exp(-s) \ ds =: \phi(k)
\end{align*}
The sequence $(\phi(k))_{k \ge 2}$ is increasing, with
$$\phi(2)= \frac 1 e \simeq 0.3679 < \phi(3) \simeq 0.3916 < \phi(4) = \frac 3 {e^2} \simeq 0.4060 \text{  and  } \lim_{k \to \infty} \phi(k) = \frac 1 2 \ . $$
\end{remark}

\begin{proof}[Proof of Proposition \ref{prop5}.]
(a) Let $S := \sum_{1 \le i < j \le n} a_i a_j <U_i,U_j>$. Since $|\sum_{j=1}^n a_j U_j|^2 = 1 + 2 S$,
$$\P(|\sum_{j=1}^n a_j U_j| \ge 1) = \P(S \ge 0) \ . $$
Since $\E(U_j) = 0$, $\E(S)=0$. By Proposition 2.3 of Veraar's paper \cite{V} on lower probability
estimates for centered random variables we have the estimate
$$\P(S \ge 0) \ge (2 \sqrt 3 -3) \ \frac {\E(S^2)^2}{\E(S^4)} \ . $$
We claim that $(3 + \frac 4 k) \ \E(S^2)^2 \ge \E(S^4)$ so that the statement of Proposition \ref{prop5} for general $k$ (not being 3 or 4)
$$\P(S \ge 0) \ge \gamma_k$$
will follow.

\bigskip
(b) To prove the claim, we calculate $\E(S^2)$ and $\E(S^4)$.
$$\E(S^2) = \sum_{i<j} \sum_{l<m} a_i a_j a_l a_m \ \E(<U_i,U_j><U_l,U_m>) \ . $$
The expectation terms on the right are non-zero only if $i=l<j=m$. Thus
$$\E(S^2) = \sum_{1 \le i < j \le n} a_i^2 a_j^2 \ \E(<U_i,U_j>^2) = (\frac 1 k \sum_{1 \le i < j \le n} a_i^2 a_j^2 \ )  $$
since  $\E(<U_i,U_j>^2) = \int_{S^{k-1}} v_1^2 \ dm(v) = \frac 1 k \int_{S^{k-1}} |v|^2 \ dm(v) = \frac 1 k $. \\

For $\E(S^4)$, we have to evaluate $\E(\prod_{l=1}^4 <U_{i_l},U_{j_l}>)$ with  $i_l < j_l$, $l=1,2,3,4$. By the
independence of the variables $U_j$, this is non-zero only if products of squares, fourth powers or cyclic combinations show up in the
index combinations, yielding cases such as
$$\E(<U_1,U_2>^2 <U_3,U_4>^2) = \E(<U_1,U_2>^2) \ \E(<U_3,U_4>^2) = \frac 1 {k^2} \ , $$
$$\E(<U_1,U_2>^2 <U_1,U_3>^2) = \E(<U_1,U_2>^2) \ \E(<U_1,U_3>^2) = \frac 1 {k^2} \ , $$
$$\E(<U_1,U_2>^4) = \int_{S^{k-1}} v_1^4 \ dm(v) = \frac{\int_0^\pi \cos(t)^4 \ \sin(t)^{k-2} \ dt}{\int_0^\pi \sin(t)^{k-2} \ dt} = \frac 3 {k(k+2)} \ , $$
or
$$\E(<U_1,U_2> <U_2,U_3> <U_3,U_4> <U_1,U_4>) = (\int_{S^{k-1}} v_1^2 \ dm(v))^3 = \frac 1 {k^3} \ . $$
Each product of squares $<U_i,U_j>^2<U_l,U_m>^2$ with $i<j$, $l<m$ and $(i,j) \neq (l,m)$ occurs
$\binom{4}{2} = 6$ times and each cyclic combination $4! = 24$ times in the fourth power expansion of $S$. Therefore
\begin{align*}
 \E(S^4)
 &= \frac 6 {k^2} (\sum_{i<j,l<m, (i,j) \neq (l,m)} a_i^2 a_j^2 a_l^2 a_m^2 \ ) + \frac 3 {k(k+2)} \sum_{i<j} a_i^4 a_j^4  \\
&+ \frac {24}{k^3} (\sum_{i<j<l<m} a_i^2 a_j^2 a_l^2 a_m^2 \ ) \ .
\end{align*}
Expanding $(\sum_{i<j} a_i^2 a_j^2)(\sum_{l<m} a_l^2 a_m^2)$, besides cases of equalities of indices, increasing index combinations show up 6 times, namely
$i<j<l<m, l<m<i<j, i<k<j<l, i<k<l<j, k<i<j<l, k<i<l<j$ which implies
$24 \ (\sum_{i<j<l<m} a_i^2 a_j^2 a_l^2 a_m^2) \le 4 \ (\sum_{i<j} a_i^2 a_j^2)^2 \ .$ Hence
\begin{align*}
\E(S^4) & \le (3+ \frac 4 k) \left( \frac 1 k \sum_{1\le i < j \le n} a_i^2 a_j^2 \right)^2 - (\frac 3 {k^2} - \frac 3 {k(k+2)}) \sum_{1 \le i < j \le n} a_i^4 a_j^4 \\
&\le (3+ \frac 4 k) \left( \frac 1 k \sum_{1 \le i <j \le n} a_i^2 a_j^2 \right)^2 = (3+ \frac 4 k) \ \E(S^2)^2 \ .
\end{align*}
This proves the claim for general $k$. To prove the better numerical estimates for $k=3, 4$, we use an Orlicz-space duality instead of the $L_2-L_2$-duality employed by Veraar.

(c) Using the Lemmas of the previous section, we now prove the better estimates for  $\P(|\sum_{j=1}^n a_j U_j|\ge 1) = \P(S \ge 0)$ in the cases $k=3, 4$.
Set
\[
 Y= \frac{S}{(\E S^2)^{1/2}}.
\]
Combining Lemmas \ref{lem: Laplace} and \ref{lem: plus bound}, we obtain
  \[
   \E \exp(\l Y_+)
   \le \frac{5}{4} \left( 1- \frac{2 \l^2 }{k} \right)^{-k/2}
  \]
for any $\l \in (0,\sqrt{k/2})$.
Also, since $\E S=0$,
\[
   \E Y_+=\frac{1}{2} \frac{\E |S|}{(\E S^2)^{1/2}}
   \ge \frac{1}{2} \left( \frac{(\E S^2)^2}{\E S^4} \right)^{1/2}
   \ge \frac{1}{2} \left( \frac{1}{3+4/k} \right)^{1/2}
   = \frac{1}{2}\sqrt{\frac{k}{3k+4}},
\]
where we used H\"older's inequality and the moment estimate from part (b).
Hence, Lemma \ref{lem: Orlicz} can be applied with
\[
  q= \frac{1}{2} \sqrt{\frac{k}{3k+4}} \quad \text{and} \quad
  t=\frac{1}{\frac{5}{4} \left( 1- \frac{2 \l^2 }{k} \right)^{-k/2}-\l q-1}.
 \]
Substituting these values in the estimate of Lemma \ref{lem: Orlicz}, and using a numerical maximization of the right hand side estimate, we obtain the desired lower bounds $0.1268$ with $\l \approx 0.7111$ for $k=3$ and $0.1407$ with $\l \approx 0.7508$ for $k=4$.
\end{proof}

\begin{remark}
At the limit $k \to \infty$, our approach yields a bound
\[
   P(S \ge 0) > 0.205475,
\]
which is about $\frac{1}{3}$ better than the bound  $\frac{2}{\sqrt{3}}-1 ~\sim 0.154700$ following   from Veraar's inequality. For $k \ge 100$, our lower bound is greater than $0.2$.
\end{remark}

We will now consider the upper tail of $|\sum_{j=1}^n a_j U_j|$. A bound for the upper tail follows directly from Lemma \ref{lem: subgauss} and the Hanson-Wright inequality. Yet, as we strive for good constants, we need a tighter estimate.

\begin{proposition}\label{prop4}
$(U_j)_{j=1}^n$ be a sequence of independent, random vectors uniformly distributed on the sphere $S^{k-1}$ for $k \in \N$, $k \ge 2$.
Let  $a \in \R_+^n, \ |a|=1$. Then for any $t>1$
$$\P(|\sum_{j=1}^n a_j U_j| \ge t) \le t^k \exp(\frac k 2 - \frac k 2 t^2) \ . $$
\end{proposition}

\begin{proof}
The Khintchine inequality for the variables $(U_j)$ states for any $p \ge 2$
$$ || \sum_{j=1}^n a_j U_j ||_{L_p(S^{k-1})} \le b_{p,k} |a| = b_{p,k} := \sqrt{\frac 2 k} \left(\frac{\Gamma(\frac{p+k}2)}{\Gamma(\frac k 2)} \right)^{\frac 1 p} \ , $$
cf. K\"onig and Kwapie\'n \cite{KKw}, Theorem 3. The constants $b_{p,k}$ are the best possible. We find for $c>0$
\begin{align*}
\int_{S^{k-1}} \exp(c |\sum_{j=1}^n a_j U_j|^2) \ d \P & = \sum_{m=0}^\infty \frac {c^m}{m!} \int_{S^{k-1}} |\sum_{j=1}^n a_j U_j|^{2m} \ d \P \\
& \le \sum_{m=0}^\infty \frac {c^m}{m!} \ b_{2m,k}^{2m} =: f_k(c) \ .
\end{align*}
We evaluate $f_k(c)$ explicitly. For $ 0 < c < \frac k 2$, we have
$$f_k(c) = \sum_{m=0}^\infty \frac {c^m}{m!} \ \left(\frac 2 k \right)^m \ \frac{\Gamma(m+\frac k 2)}{\Gamma(\frac k 2)} =
\sum_{m=0}^\infty \binom{-\frac k 2}{m} \left(-\frac {2c}m \right)^m = \left(1- \frac{2c} k \right)^{-\frac k 2} \ .$$
Therefore for any fixed $t>1$
$$\P(|\sum_{j=1}^n a_j U_j| \ge t) \exp(c t^2) \le \int_{S^{k-1}} \exp(c |\sum_{j=1}^n a_j U_j|^2) \ d \P = f_k(c) \ , $$
$$\P(|\sum_{j=1}^n a_j U_j| \ge t) \le f_k(c) \exp(-ct^2) =: g_k(c) \ .$$
For a given $t>1$, $g_k$ is minimal for $\bar{c} = \frac k 2 (1-\frac 1 {t^2})$. This yields
$$\P(|\sum_{j=1}^n a_j U_j| \ge t) \le t^k \exp(\frac k 2 - \frac k 2 t^2) \ . $$
\end{proof}

\section{A lower bound for hyperplane sections} \label{sec: hyperplane}
In this section, we prove Theorem \ref{th1}.

 We first consider the real case. By Proposition \ref{prop3}
$$A_\R(a,1) = \int \limits_{|\sum_{j=1}^n a_j U_j| \ge 1} \; \frac {d  \P}{|\sum_{j=1}^n a_j U_j|} \ , $$
where the $(U_j)_{j=1}^n$ are independent random vectors uniformly distributed on the sphere $S^2 \subset \R^3$. Hence
\begin{align*}
A_\R(a,1) &=  \int_0^1 \ \P( 1 \le |\sum_{j=1}^n a_j U_j| < \frac 1 s ) \ ds \\
& = \int_1^\infty \frac{\P( 1 \le |\sum_{j=1}^n a_j U_j| < v )}{v^2} \ dv \\
& \ge \int_t^\infty \frac{\P( 1 \le |\sum_{j=1}^n a_j U_j| < v )}{v^2} \ dv
\end{align*}
for any $t > 1$. By Proposition \ref{prop4} for $k=3$, $\P( |\sum_{j=1}^n a_j U_j| \ge v ) \le v^3 \exp(\frac 3 2 - \frac 3 2 v^2)$, and by Proposition \ref{prop5} for $k=3$, $\P( |\sum_{j=1}^n a_j U_j| \ge 1 ) \ge 0.1268 := p_0$. Choose $t_0 > 1$ such that  $t_0^3 \exp(\frac 3 2 - \frac 3 2 t_0^2) = p_0$, $t_0 \simeq 1.9182$. Then for $t=t_0$
\begin{align*}
A_\R(a,1) & \ge \int_{t_0}^\infty \frac{\P( |\sum_{j=1}^n a_j U_j| \ge 1 ) - \P( |\sum_{j=1}^n a_j U_j| \ge v )}{v^2} \ dv \\
& \ge \frac{\P( |\sum_{j=1}^n a_j U_j| \ge 1 )}{t_0} - \int_{t_0}^\infty v \exp(\frac 3 2 - \frac 3 2 v^2) \ dv \\
& \ge \frac{p_0}{t_0} - \frac 1 3 \exp(\frac 3 2 - \frac 3 2 {t_0}^2) = \frac{p_0}{t_0} (1- \frac 1 {3 t_0^3}) > 0.06011 > \frac 3 {50} > \frac 1 {17} \ ,
\end{align*}
which is the claim of Theorem \ref{th1} in the real case. \\
\vskip 0.1in

 In the complex case, by Proposition \ref{prop3}
$$A_\C(a,1) = \int \limits_{|\sum_{j=1}^n a_j U_j| \ge 1} \; \frac {d  \P}{|\sum_{j=1}^n a_j U_j|^2}\ ,$$
where the $(U_j)_{j=1}^n$ now are independent uniformly distributed random vectors on $S^3 \subset \R^4$. Thus
\begin{align*}
A_\C(a,1) &= \int_0^1  \P( 1 \le |\sum_{j=1}^n a_j U_j| < \frac 1 {\sqrt s} ) \ ds \\
& = 2 \ \int_1^\infty \frac{\P( 1 \le |\sum_{j=1}^n a_j U_j| < v )}{v^3} \ dv  \\
& \ge 2 \ \int_t^\infty \frac{\P( 1 \le |\sum_{j=1}^n a_j U_j| < v )}{v^3} \ dv
\end{align*}
for any $t>1$. By Propositions \ref{prop4} and \ref{prop5} for $k=4$, $\P( |\sum_{j=1}^n a_j U_j| \ge v ) \le v^4 \exp(2-2 v^2)$ and
$\P( |\sum_{j=1}^n a_j U_j| \ge 1 ) \ge 0.1407 := p_1$. Choose $t_1 > 1$ with $t_1^4 \exp(2-2 t_1^2) = p_1$, $t_1 \simeq 1.7657$. Then for $t=t_1$
\begin{align*}
A_\C(a,1) &\ge \frac{\P(|\sum_{j=1}^n a_j U_j| \ge 1)}{t_1^2} - 2 \int_{t_1}^\infty v \exp(2 - 2 v^2) \ dv \\
&\ge \frac{p_1}{t_1^2} - \frac 1 2 \exp(2 - 2 t_1^2) = \frac{p_1}{t_1^2} (1 - \frac 1 {2 t_1^2} ) > 0.03789 > \frac 1 {27} ,
\end{align*}
which proves Theorem \ref{th1} also in the complex case of the polydisc sections. \\

\begin{remark} \label{rem: diagonal}
The estimates of Theorem \ref{th1} cannot be improved by more than a factor of $\simeq 5.1$ in the real case and by a factor of $\simeq 7.1$ in the complex case.
Indeed, consider the diagonal directions.
Let $a^n:= \frac 1 {\sqrt n} (1,\cdots,1) \in \R^n$, $|a^n|=1$. For $n=2,3$ the vectors $a^2 \in \R^2$ and $a^3 \in \R^3$ yield the minimal values of hyperplane sections $A_\R(a,1)$, $|a|=1$ in $Q_2$ and in $Q_3$,
$$A_\R(a^2,1) = \sqrt 2 - 1 \simeq 0.4142 > A_\R(a^3,1) = \frac {6 \sqrt 3 - 9} 4 \simeq 0.3481 ,$$
cf. K\"onig, Koldobsky \cite{KK1}. It is unclear whether $A(a^n,1)$ provides the minimal value of hyperplane section volumes in $Q_n$ for $n>3$. Actually, the sequence $(A(a^n,1))_{n=2}^\infty$ is decreasing with
$$\lim_{n \to \infty} A_\R(a^n,1) = \frac 2 \pi \int_0^\infty \exp(-\frac{s^2}6) \ \cos(s) \ ds = \sqrt{\frac 6 {\pi e^3}} \simeq 0.3084 \ . $$
Therefore no improvement of the lower bound beyond $\sqrt{\frac 6 {\pi e^3}} \simeq 5.1 \cdot 0.06011$ is possible in the real case. In the complex case
$$\lim_{n \to \infty} A_\C(a^n,1) = \frac 1 2 \int_0^\infty \exp(-\frac{s^2}8) \ J_0(s) \ s \ ds = \frac 2 {e^2} \simeq 0.2707 \ .$$
\end{remark}

\end{document}